\documentclass{article}
\usepackage{graphicx} 
\usepackage{amsfonts}
\usepackage{amsmath}
\usepackage{amsthm}
\usepackage{amssymb}
\usepackage{xcolor}
\usepackage{url}
\usepackage{orcidlink}
\usepackage{hyperref}

\newtheorem{theorem}{Theorem}[section]
\newtheorem{lemma}{Lemma}[section]
\newtheorem{proposition}{Proposition}[section]

\hypersetup{
  pdfauthor={Lachlan Dunn (Orcid: 0009-0007-8100-1011)},
  pdftitle={Improved Upper Bound on Brun's Constant Under GRH},
  pdfsubject={2020 MSC: 11Y60, 11N36},
  pdfkeywords={Brun's constant, number theory, twin primes, sieve methods, analytic number theory}
}

\title{Improved Upper Bound on Brun's Constant Under GRH}
\author{Lachlan Dunn\footnote{\noindent Email: lachlan.dunn1@student.uq.edu.au\\ Affiliation: School of Mathematics and Physics, University of Queensland} }
\date{\today}

\begin{document}

\maketitle

\begin{abstract}
    Brun's constant is the summation of the reciprocals of all twin primes, given by $B=\sum_{p \in P_2}{\left( \frac{1}{p} + \frac{1}{p+2}\right)}$. While rigorous unconditional bounds on $B$ are known, we present the first rigorous bound on Brun's constant under the GRH assumption, yielding $B < 2.1594$.
\end{abstract}

\section{Introduction}
Brun's constant is a fundamental concept in analytic number theory, arising in the study of twin primes—pairs of prime numbers that differ by two. That is, $P_2 := \{\text{prime }p:p+2 \text{ is prime}\}$. Unlike the sum of the reciprocals of all prime numbers, which diverges, the sum of the reciprocals of twin primes converges to a finite value, known as Brun’s constant. Perhaps unsurprisingly, this result was first established in 1919 by Brun \cite{Brun1919}. 

Brun’s proof, however, did not yield a numerical upper bound for $B$. The first such bound, $B < 2.347$, was established by Crandall and Pomerance \cite{Crandall2006}. Since then, improvements have been made, with the current best unconditional bound of $B < 2.28851$ being obtained by Platt and Trudgian \cite{Platt2020}. Nicely conjectured that $B = 1.902160583209 \pm 0.000000000781$ \cite{Nicely2010}, demonstrating the scope of potential improvements to be made.

One approach to refining these bounds is to assume the Generalized Riemann Hypothesis (GRH). This method was employed by Klyve to derive the bound $B < 2.1754$ \cite{DominicKlyve2007}, and it will be the approach taken here. Although our improvement is small, Klyve never published his result, and utilized numerical integration to calculate his subsequent bound on $B$. To our knowledge, our work represents the first mathematically rigorous upper bound on $B$ assuming GRH.

\begin{theorem}
Assume GRH. Then, $B < 2.1594$.
\end{theorem}

To achieve this, we combine previously computed values of $B$ (up to a fixed threshold $x_0$), along with the estimation of the tail end of the sum provided by Selberg’s sieve. This estimation proceeds by computing the intermediate sums
\begin{equation*}
    B(m_i,m_{i+1}):=\sum_{\substack{p \in P_2 \\ m_i<p\leq m_{i+1}}}{\left( \frac{1}{p} + \frac{1}{p+2}\right)}
\end{equation*}
where $m_1 = x_0$ and $i=1,\ldots k$. These sums are then combined to evaluate $B(m_1,\infty)$.

This intermediate sum approach is advantageous due to its flexibility, as different methods of estimating $B$ may yield tighter bounds over different regions. By combining these results, a sharper overall bound can be obtained. To facilitate further improvements, we provide explicit bounds for each $B(m_i,m_{i+1})$.

The structure of this paper is as follows: Section 2 introduces the sieve framework, defines key notation, and derives a bound for the twin prime counting function. Section 3 extends this bound to Brun’s constant using partial summation. Section 4 focuses on parameter optimization and presents the numerical computations. Finally, Section 5 discusses potential refinements and further improvements to this method.

\section{Upper Bound on $\pi_2(x)$}
To calculate $B$, we require an upper bound on the twin prime counting function $\pi_2(x) := \#\{p \text{ prime}: p+2 \text{ also prime}\}$. We derive this bound using Selberg's sieve by sieving on the set of odd numbers $n$ such that $n-2$ is prime. Notably, if $n$ is removed by the sieve (i.e. $n$ is not prime), then $(n-2,n)$ cannot form a twin prime pair.
\subsection{Preliminaries}
For an introduction to Selberg's sieve, see \cite{Halberstam1974}. We will begin by defining below some preliminary notation. Our set of twin prime candidates is given by
\begin{equation*}
    \mathcal{A} = \{n \leq x+2: n\text{ odd with }n-2 \text{ prime}\}
\end{equation*}
and our set of sieving primes by 
\begin{equation*}
    \mathcal{P}=\{p>2:p\text{ prime}\}.
\end{equation*}
We will be sieving using odd primes less than or equal to $z$, and so 
\begin{equation*}
    P_z= \prod_{\substack{p \in\mathcal{P}\\ p\leq z}}{p}.
\end{equation*}
$\mathcal{A}_d$ is the set of twin prime candidates that are divisible by a specific prime $d$, that is
\begin{equation*}
    A_d= |\{n \leq x+2:d|n \text{ and } n-2 \text{ is prime}\}| = \pi(x;d,-2).
\end{equation*}
By the prime number theorem in arithmetic progressions, we will be estimating $\pi(x;d,-2)$ by $\frac{\text{li}(x)}{\varphi(d)}$, and so we define the error in our estimation to be
\begin{equation*}
    E(x;d):= \pi(x;d,-2) - \frac{\text{li}(x)}{\varphi(d)}
\end{equation*}
where
\begin{equation*}
    \text{li}(x) := \int_2^x\frac{dt}{\log{t}}.
\end{equation*}
As we take li to be truncated at 2, it will be useful to define the offset as
\begin{equation*}
    \text{li}_2:=\int_0^2\frac{dt}{\log{t}}.
\end{equation*}
We will use the following notation for the twin prime constant: 
\begin{equation*}
    \pi_2 := \prod_{p>2}\frac{p(p-2)}{(p-1)^2}.
\end{equation*} Following Klyve\footnote{The relevant working is found at page 128.} \cite{DominicKlyve2007}, we take
\begin{equation*}
    V(z) = \sum_{\substack{d \leq z \\ d|P_z}}{\frac{1}{\prod_{p|d}{(p-2)}}}.
\end{equation*}
And finally, we take
\begin{equation*}
    [d_1,d_2]:=\text{lcm}(d_1,d_2).
\end{equation*}
Then, using Theorem 5.7 \cite{DominicKlyve2007} gives us the following result.
\begin{theorem}\mbox{}\\
For any $0<z<x$,
    \begin{equation*}
        \pi_2(x) \leq |\mathcal{A} \setminus \bigcup_{p|P_z} \mathcal{A}_p| = S(\mathcal{A},\mathcal{P}, z) \leq \frac{\textup{li}(x)}{V(z)}+ R(x,z)
    \end{equation*}
where we have defined
    \begin{equation*}
        R(x,z):= \sum_{\substack{d_1, d_2 \leq z \\ d_1,d_2 | P_z}}{|E(x;[d_1,d_2])|}.
    \end{equation*}
\end{theorem}

\subsection{Bounding $V^{-1}(z)$}
We need to provide an upper bound on $V^{-1}(z)$ for all $z$. As $z$ increases, computing $V(z)$ directly becomes infeasible, requiring a tradeoff betwen accuracy and efficiency. To address this, we will consider three separate intervals (separated by values $L_1$ and $L_2$), with different bounds on $V(z)$ for each interval. For our calculations, we set $L_1 = 10^8$ and $L_2 = 10^{10}$.\\
For the first interval ($z < L_1$), we will compute $V(z)$ exactly. For the second interval ($L_1 < z \leq L_2$), we will compute $V(z')$ exactly, where $z' \in \{L_1<z\leq L_2: z=k\cdot 10^4, k \in \mathbb{N}\}$ (that is, every $10^4$th number). For the values of $z$ such that $10^4 \nmid z$, we can round $z$ down to the nearest $z'$ to give an upper bound on $V^{-1}(z)$. For the third interval, we will use the approximation of $V(z)$ derived in Section 5.5.2 \cite{DominicKlyve2007}, namely
\begin{equation*}
    V(z) \geq V(L_2) + \frac{\log{z}-\log{L_2}}{2\pi_2} + D(L_2,z)
\end{equation*}
where
\begin{equation*}
    D(L_2,z) = 12.6244\left(\frac{3}{\sqrt{z}} - \frac{9}{\sqrt{L_2}}\right).
\end{equation*}
To summarise,
\begin{lemma} For any $z$, we have\\
$V(z) \geq $
$\begin{cases}
    V(L_2) + \frac{\log{z}-\log{L_2}}{2\pi_2} + D(L_2,z) & z > L_2\\
    V(z') & L_1 <z \leq L_2 \\
    V(z) & z< L_1
\end{cases}$.
\end{lemma}

\subsection{Bounding $R(x,z)$}
Now, we need an upper bound on $R(x,z)$ for all $x$ and $z$. We will start by proving an inequality that will allow us to bound one of $R(x,z)$'s terms. This inequality holds for any fixed prime $p$, but becomes less useful as the size of $p$ increases relative to the size of $z$. We have excluded some slight optimisations on the order of $z^{1/2}$ for the sake of readability, and the range of permissible values of $z$ can be lowered at a similar cost. For our case, we will be setting $p=2$. 
\begin{proposition}\label{Thm:R}
    Let $p$ be a prime and $z \geq 2768896$. Then, 
    \begin{align*}
        \sum_{\substack{d \leq z \\ \textup{gcd}(p,d)=1}}{\mu^2(d)} \leq &\left(\frac{p}{p+1}\cdot \frac{6}{\pi^2}\right)z + \left( \frac{6}{\pi^2}(1+\frac{1}{\sqrt{p}})+2\right)\sqrt{z} \\&+ \left(3.12 \cdot(1+\frac{1}{p^{1/4}})\right)z^{1/4}.
    \end{align*}
\end{proposition}
\begin{proof}
    To start, note that we can write
    \begin{align*}
        \sum_{\substack{d \leq z \\ \textup{gcd}(p,d)=1}}{\mu^2(d)} &= \sum_{\substack{k \leq \sqrt{z} \\ \textup{gcd}(p,k)=1}}\mu(k)\left[\frac{z}{k^2}\right] - \sum_{\substack{k \leq \sqrt{\frac{z}{p}} \\ \textup{gcd}(p,k)=1}}\mu(k)\left[\frac{z}{pk^2}\right] \\
        &\leq \sum_{\substack{k \leq \sqrt{z} \\ \textup{gcd}(p,k)=1}}\mu(k)\frac{z}{k^2} + \sum_{\substack{k \leq \sqrt{z} \\ \textup{gcd}(p,k)=1}}1 - \sum_{\substack{k \leq \sqrt{\frac{z}{p}} \\ \textup{gcd}(p,k)=1}}\mu(k)\frac{z}{pk^2} \\&\quad+\sum_{\substack{k \leq \sqrt{\frac{z}{p}} \\ \textup{gcd}(p,k)=1}}1\\
        &\leq \sum_{\substack{k \leq \sqrt{z} \\ \textup{gcd}(p,k)=1}}\mu(k)\frac{z}{k^2} - \sum_{\substack{k \leq \sqrt{\frac{z}{p}} \\ \textup{gcd}(p,k)=1}}\mu(k)\frac{z}{pk^2} + 2 \sqrt{z}\\
        &= M(z)z + 2 \sqrt{z},
    \end{align*}
    where we have taken
    \begin{equation*}
        M(z) :=\sum_{\substack{k \leq \sqrt{z} \\ \textup{gcd}(p,k)=1}}\frac{\mu(k)}{k^2} - \sum_{\substack{k \leq \sqrt{\frac{z}{p}} \\ \textup{gcd}(p,k)=1}}\frac{\mu(k)}{pk^2}.
    \end{equation*}
    Also, we have the following, where $p'$ is any prime number.
    \begin{align*}
        M(z) + \sum_{\substack{k > \sqrt{z} \\ \textup{gcd}(p,k)=1}}\frac{\mu(k)}{k^2} - \sum_{\substack{k > \sqrt{\frac{z}{p}} \\ \textup{gcd}(p,k)=1}}\frac{\mu(k)}{pk^2} &= \sum_{\substack{k=1 \\ \textup{gcd}(p,k)=1}}^{\infty}\frac{\mu(k)}{k^2} - \sum_{\substack{k =1 \\ \textup{gcd}(p,k)=1}}^{\infty}\frac{\mu(k)}{pk^2} \\ &=(1-\frac{1}{p})\sum_{\substack{k=1 \\ \textup{gcd}(p,k)=1}}^{\infty}\frac{\mu(k)}{k^2} \\ 
        &=(1-\frac{1}{p})\prod_{p' \neq p}(1-\frac{1}{(p')^2})\\
        &=\frac{(1-\frac{1}{p})}{(1-\frac{1}{p^2})} \cdot \frac{6}{\pi^2} \\
        &=\frac{p}{p+1} \cdot \frac{6}{\pi^2}
    \end{align*}
    Now we can bound the difference as follows.
    \begin{align*}
        \left|\sum_{\substack{k > \sqrt{z} \\ \textup{gcd}(p,k)=1}}\frac{\mu(k)}{k^2} -  \sum_{\substack{k > \sqrt{\frac{z}{p}} \\ \textup{gcd}(p,k)=1}}\frac{\mu(k)}{pk^2} \right| \leq \sum_{k > \sqrt{z}}\frac{\mu^2(k)}{k^2} +  \frac{1}{p}\sum_{k > \sqrt{\frac{z}{p}} }\frac{\mu^2(k)}{k^2}.\\
    \end{align*}
    We now apply explicit bounds on the square-free summation function provided in \cite{Cohen1985}, in conjunction with partial summation, to note that
    \begin{align*}
        \sum_{k > x}\frac{\mu^2(k)}{k^2} \leq \left(-\frac{\frac{6}{\pi^2}x - 1.333\sqrt{x}}{x^2} + 2\int_{x}^{\infty}\frac{\frac{6}{\pi^2}t + 1.333\sqrt{t}}{t^3}dt \right).
    \end{align*}
    Applying this substitution, and simplifying terms, we conclude that
    \begin{align*}
        \left|\sum_{\substack{k > \sqrt{z} \\ \textup{gcd}(p,k)=1}}\frac{\mu(k)}{k^2}-  \sum_{\substack{k > \sqrt{\frac{z}{p}} \\ \textup{gcd}(p,k)=1}}\frac{\mu(k)}{pk^2} \right| &\leq \frac{\frac{6}{\pi^2}}{\sqrt{z}} + \frac{3.12}{z^{3/4}} + \frac{1}{p}\left(\frac{\frac{6}{\pi^2}}{\sqrt{\frac{z}{p}}} + \frac{3.12}{(\frac{z}{p})^{3/4}} \right)\\
        &= \frac{\frac{6}{\pi^2}}{\sqrt{z}} + \frac{3.12}{z^{3/4}} + \frac{\frac{6}{\pi^2}}{\sqrt{zp}} + \frac{3.12}{z^{3/4}p^{1/4}}.
    \end{align*}
    Combining this result with the first inequality proves the lemma.
\end{proof}
\noindent In a similar manner to $V(z)$, $\sum_{\substack{d \leq z \\ \textup{gcd}(2,d)=1}}{\mu^2(d)}$ becomes infeasible to compute for large values of $z$. We employ a similar technique of approximating $\sum_{\substack{d \leq z \\ \textup{gcd}(2,d)=1}}{\mu^2(d)}$ for $z$ larger than a specific cut-off value $L_3$ (using Proposition~\ref{Thm:R}), and calculating it exactly for the remaining values. We take $L_3=10^7$ for our calculations. 
\begin{lemma}\label{Lemma:r}
    For any $z$, we have
    \begin{align*}
        &\sum_{\substack{d \leq z \\ \textup{gcd}(2,d)=1}}{\mu^2(d)} \leq r(z):= 
        \begin{cases}
            \sum_{\substack{d \leq z \\ \textup{gcd}(2,d)=1}}{\mu^2(d)} &z \leq L_3 \\
            \frac{4}{\pi^2}z + 3.038\sqrt{z} + 5.744z^{1/4} &z>L_3
        \end{cases}.
    \end{align*}
\end{lemma}
\noindent Now we continue with the following.
\begin{theorem}\mbox{}\\
Assuming GRH, for $x \geq 4\cdot 10^{18}$ and $z<x^{\frac{1}{4}}$,
\begin{equation*}
    \sum_{\substack{d_1, d_2 \leq z \\ d_1,d_2 | P_z}}{|E(x;[d_1,d_2])|} \leq r^2(z)c_\pi(x)\sqrt{x}\log{x}
\end{equation*}
where 
\begin{align*}
    &c_\pi(x) := \frac{3}{8\pi} + \frac{6+1/\pi}{4\log{x}} + \frac{6}{\log^2{x}} + \frac{\textup{li}_2}{\sqrt{x}\log{x}}.
\end{align*}
\end{theorem}
\begin{proof}
    We will use Lemma 5.1 \cite{Johnston2025} to bound under GRH, noting that for $z<x^{\frac{1}{4}}$, $[d_1,d_2]<d_1d_2 < z^2 < \sqrt{x}$. Thus, for $x \geq 4\cdot 10^{18}$ and $z<x^{\frac{1}{4}}$, we have
    \begin{align*}
        \sum_{\substack{d_1, d_2 \leq z \\ d_1,d_2 | P_z}}{|E(x;[d_1,d_2])|} &< \sum_{\substack{d_1, d_2 \leq z \\ d_1,d_2 | P_z}}{c_\pi(x)\sqrt{x}\log{x}}  \\ 
        &= c_\pi(x)\sqrt{x}\log{x}\sum_{\substack{d_1, d_2 \leq z \\ d_1,d_2 | P_z}}{1} \\
        &=c_\pi(x)\sqrt{x}\log{x}\left(\sum_{\substack{d \leq z \\ d| P_z}}{1}\right)^2 \\&=c_\pi(x)\sqrt{x}\log{x}\left(\sum_{\substack{d \leq z \\ \gcd(2,d)=1}}{\mu^2(d)}\right)^2 \\
        &=r^2(z)c_\pi(x)\sqrt{x}\log{x},
    \end{align*}
    by applying Lemma~\ref{Lemma:r} on the second last line.
\end{proof}

\section{Upper Bound on $B$} 
Using Lemma 6 \cite{Platt2020}, we have that
\begin{align*}
    B(2,4\cdot 10^{18})\leq 1.840518.
\end{align*}
Now, using our bound on $\pi_2(x)$, we can apply partial summation to calculate $B(m_1,\infty)$. Let $z_i$ be fixed for each interval $[m_i,m_{i+1}]$ (the values of each $z_i$ and the intervals $[m_i,m_{i+1}]$ will be determined later).
\begin{lemma}
    Assuming GRH, for $m_1 \geq 4\cdot 10^{18}$ and $z_i<m_i^{\frac{1}{4}}$,
    \begin{align*}
        B(m_k)-B(m_1) < &2 \sum_{i=1}^{k-1}\left(\frac{c_1(m_i,m_{i+1})}{V(z_i)} + c_\pi(m_i)r^2(z_i)c_2(m_i,m_{i+1}) \right) \\
        &+2\left(\frac{\pi_2(m_k)}{m_k} -\frac{\pi_2(m_1)}{m_1} \right),
    \end{align*}
    where 
    \begin{align*}
        &\textup{li}_{\textup{up}}(x) := \frac{x}{\log{x}}\left(1 + \frac{1}{\log{x}}  + \frac{2}{\log^2{x}} + \frac{7.32}{\log^3{x}} \right) + \frac{\sqrt{x}\log{x}}{8\pi}  + \textup{li}_2, \\
        &\textup{li}_{\textup{low}}(x) := \frac{x}{\log{x}}\left(1 + \frac{1}{\log{x}}  + \frac{2}{\log^2{x}} \right) - \frac{\sqrt{x}\log{x}}{8\pi}  - \textup{li}_2, \\
        & c_1(m_i,m_{i+1}) := \left(\log{\log{m_{i+i}}}- \frac{\textup{li}_{\textup{low}}(m_{i+1}) - \textup{li}_2}{m_{i+1}}\right) \\& \qquad \qquad \qquad \quad -\left(\log{\log{m_{i}}}- \frac{\textup{li}_{\textup{up}}(m_{i}) - \textup{li}_2}{m_{i}}\right),\\
        & c_2(m_i,m_{i+1}) := \left( \frac{-2\log{m_{i+1}}-4}{\sqrt{m_{i+1}}} \right) - \left( \frac{-2\log{m_i}-4}{\sqrt{m_i}} \right).
    \end{align*}
\end{lemma}
\begin{proof}
    \begin{align*}
        B(m_k)-B(m_1) &= \sum_{\substack{m_1<p\leq m_k \\ p,p+2 \text{ prime}}}{\left(\frac{1}{p} + \frac{1}{p+2}\right)}\\
        &< 2\sum_{\substack{m_1<p\leq m_k \\ p,p+2 \text{ prime}}}{\frac{1}{p}} \\
        &=2\left(\frac{\pi_2(m_k)}{m_k} -\frac{\pi_2(m_1)}{m_1} + \int_{m_1}^{m_k}{\frac{\pi_2(t)}{t^2}}dt\right) \\
        &=2\sum_{i=1}^{k-1}\left( \int_{m_i}^{m_{i+1}}{\frac{\pi_2(t)}{t^2}}dt \right) + 2\left(\frac{\pi_2(m_k)}{m_k} -\frac{\pi_2(m_1)}{m_1} \right).
    \end{align*}
    The value of $\frac{\pi_2(m_k)}{m_k}$ will cancel in a further calculation, and the value of $\frac{\pi_2(m_1)}{m_1}$ can be calculated exactly using the computations performed by Oliveira e Silva \cite{OliveiraeSilva2015}. For the remaining term, we proceed by
    \begin{align*}
        \hspace*{0pt} 2 \sum_{i=1}^{k-1}&\left( \int_{m_i}^{m_{i+1}}{\frac{\pi_2(t)}{t^2}} dt\right) \\ &< 2 \sum_{i=1}^{k-1}\left( \int_{m_i}^{m_{i+1}}{\frac{\textup{li}(t)}{V(z_i)\cdot t^2}}dt + \int_{m_i}^{m_{i+1}}{\frac{ r^2(z_i)c_\pi(t)\sqrt{t}\log{t}}{t^2}}dt \right).
    \end{align*}
    We now apply some simplifications to these integrals to allow us to compute them analytically. Firstly, note that $V(z_i)$ and $r(z_i)$ are fixed for each interval $[m_i,m_{i+1}]$. Also, note that $c_{\pi}(t)$ is monotonically decreasing, and so $c_{\pi}(t) \leq c_{\pi}(m_i)$ for $x \in [m_i,m_{i+1}]$. And so,
    
    \begin{align*}
        B(m_k)-B(m_i)<&2 \sum_{i=1}^{k-1}\left( V(z_i)^{-1}\int_{m_i}^{m_{i+1}}{\frac{\textup{li}(t)}{t^2}}dt \right. \\&+\left. c_\pi(m_i)r^2(z_i)\int_{m_i}^{m_{i+1}}{\frac{ \sqrt{t}\log{t}}{t^2}}dt \right).
    \end{align*}
    Now, in computing the integral we note that 
    \begin{align*}
        V(z_i)^{-1}\int_{m_i}^{m_{i+1}}{\frac{\text{li}(t)}{t^2}}dt = &V(z_i)^{-1}\left [\left(\log{\log{m_{i+1}}}- \frac{\text{li}(m_{i+1}) - \text{li}_2}{m_{i+1}}\right) \right. \\ & \left. - \left(\log{\log{m_{i}}}- \frac{\text{li}(m_{i}) - \text{li}_2}{m_{i}}\right)\right].
    \end{align*}
    To avoid computing li$(x)$ directly, we use Schoenfeld's bound (6.18) on $|\pi(x)-\text{li}(x)|$ assuming RH \cite{Schoenfeld1976}, and Dusart's bounds on $\pi(x)$ (Theorem 5.1 and Corrollary 5.2) \cite{Dusart2018}, to find (for $x \geq 4 \cdot 10^9$)
    \begin{align*}
        &\text{li}(x) \leq \text{li}_{\text{up}}(x) = \frac{x}{\log{x}}\left(1 + \frac{1}{\log{x}}  + \frac{2}{\log^2{x}} + \frac{7.32}{\log^3{x}} \right) + \frac{\sqrt{x}\log{x}}{8\pi}  + \text{li}_2 \\
        &\text{li}(x) \geq \text{li}_{\text{low}}(x) = \frac{x}{\log{x}}\left(1 + \frac{1}{\log{x}}  + \frac{2}{\log^2{x}} \right) - \frac{\sqrt{x}\log{x}}{8\pi}  - \text{li}_2 .
    \end{align*}
    And so
    \begin{align*}
        \int_{m_i}^{m_{i+1}}{\frac{\text{li}(t)}{t^2}}dt \leq &\left(\log{\log{m_{i+i}}}- \frac{\text{li}_{\text{low}}(m_{i+1}) - \text{li}_2}{m_{i+1}}\right) - \\& \left(\log{\log{m_{i}}}- \frac{\text{li}_{\text{up}}(m_{i}) - \text{li}_2}{m_{i}}\right).
    \end{align*}
    And finally,
    \begin{align*}
        c_\pi(m_i)r^2(z_i)\int_{m_i}^{m_{i+1}}{\frac{ \sqrt{t}\log{t}}{t^2}}dt =&c_\pi(m_i)r^2(z_i) \left[\left( \frac{-2\log{m_{i+1}}-4}{\sqrt{m_{i+1}}} \right) \right. \\ & \left. - \left( \frac{-2\log{m_i}-4}{\sqrt{m_i}} \right)\right].
    \end{align*}
    Combining the above equations gives the desired result.
\end{proof}

\begin{lemma}
    \begin{equation}
        B - B(m_k) < \frac{16 \pi_2}{\log{m_k}} + \frac{8}{\sqrt{m_k}}.
    \end{equation}
\end{lemma}
\begin{proof}
    We use Lemma 3 \cite{Platt2020} to bound the tail end of $B$, so that
    \begin{align*}
        B - B(m_k) &\leq -2 \frac{\pi_2(m_k)}{m_k} + \int_{m_k}^{\infty}{\frac{32 \pi_2}{t\log^2{t}} + 4t^{-\frac{3}{2}}dt}\\&=-2 \frac{\pi_2(m_k)}{m_k} +\frac{32 \pi_2}{\log{m_k}} + \frac{8}{\sqrt{m_k}}.
    \end{align*}
\end{proof}
\noindent Combining the above two lemmas with Lemma 6 \cite{Platt2020} gives the following result.
\begin{theorem}
    \begin{align*}
        B \leq &1.840518 + 2 \sum_{i=1}^{k-1}\left(\frac{c_1(m_i,m_{i+1})}{V(z_i)} + c_\pi(m_i)r^2(z_i)c_2(m_i,m_{i+1}) \right) + \frac{32 \pi_2}{\log{m_k}} \\& + \frac{8}{\sqrt{m_k}} -2\frac{\pi_2(m_1)}{m_1}.
    \end{align*}
\end{theorem}

\section{Calculating $B$}
The final step is to determine the optimal parameters for the theorem. This involves selecting appropriate intervals $[m_i,m_{i+1}]$, and choosing corresponding values of $z_i$ for each interval. Generally, increasing the number of intervals reduces the overall error, though with diminishing returns as the interval size decreases. Based on this principle, we set $m_i \in \{4 \cdot 10^{18}\} \cup \{10^i: i = 19, 19.2, ..., 1999.8, 2000\}$. To determine $z_i$, we employ the Optim \cite{Mogensen2018} package in Julia to optimize $z_i$ within a specified domain. Since the optimal value of $z_i$ increases with $i$, we can restrict the search domain accordingly.
Finally, all of the computations for $B$ were performed using the IntervalArithmetic package in Julia \cite{Sanders2014}, ensuring rigorous explicit bounds.

A subset of the optimized parameter ranges is presented in Table~\ref{tab:basic}. The full set of intermediate calculations, along with the optimized values of $z_i$ for each interval, is available in the linked GitHub repository\footnote{\url{https://github.com/LachlanJDunn/B_Under_GRH}}. This repository also includes the functions used to compute $B$ and perform the parameter sweep. Additionally, it provides a method for estimating bounds on $B(a,b)$ for any $4\cdot10^{18} \leq m_1 \leq a < b \leq m_k$, by computing $B(m_i,m_j)$ for the smallest interval that fully contains $[a,b]$.

\begin{table}[h]
    \centering
    \begin{tabular}{|c|c|c|c|}
        \hline
        Interval & $B(m_i,m_{i+1})$ & Range of $z_i$ values\\ 
        \hline
        $[4\cdot 10^{18},10^{20}]$ & $0.0256$ & $[725,1263]$\\ 
        $[10^{20},10^{25}]$ & $0.0662$ & $[1377, 13948]$\\ 
        $[10^{25},10^{30}]$ & $0.0423$ & $[15598,173542]$\\  
        $[10^{30},10^{40}]$ & $0.0506$ & $[191170,3.02\cdot10^7]$\\
        $[10^{40},10^{50}]$ & $0.0290$ & $[3.35\cdot10^7,6.01\cdot10^9]$\\
        $[10^{50},10^{100}]$ & $0.0547$ & $[6.69\cdot10^9,4.53\cdot10^{21}]$\\
        $[10^{100},10^{2000}]$ & $0.0471$ &  $[5.06\cdot10^{21},1.04\cdot10^{494}]$ \\
        \hline
    \end{tabular}
    \caption{$B(m_i,m_{i+1})$ for chosen intervals (rounded to 4 significant figures)}
    \label{tab:basic}
\end{table}

\section{Potential Improvements}
We conclude by discussing several potential improvements to the presented method of bounding $B$. 
\begin{itemize}
    \item Combining different bounds: as discussed earlier, splicing together different bounds for $B$ over disjoint intervals could provide stronger estimates than restricting to only one method. To facilitate this, the intermediate sums computed in this paper are made available.
    \item Evaluating $B(x_0)$ at a higher $x_0$: computing $B(x_0)$ for a larger $x_0$ would reduce approximation errors, leading to a more precise bound. 
    \item Further parameter optimisation: refining the choices of parameters $\{m_i\}$ and $\{z_i\}$ could potentially yield a more accurate result.
\end{itemize}

\section{Acknowledgements}
This work was completed during and funded by the School of Mathematics and Physics' Summer Research program at the University of Queensland. Firstly, I would like to thank Adrian Dudek for supervising me throughout the program, and for all his help during and after the program ended. I would also like to thank Daniel Johnston and Dave Platt for their very helpful discussions, and the referee for their comments.

\bibliographystyle{abbrvurl}
\bibliography{library}

\end{document}